\documentclass[12pt,a4paper]{article}

\usepackage[utf8]{inputenc}
\usepackage[T1]{fontenc}
\usepackage[ngerman,english]{babel}
\usepackage{lmodern}

\usepackage{amsmath}
\usepackage{amsthm}
\usepackage{amsfonts}
\usepackage{amssymb}
\usepackage{geometry}
\usepackage{color}
\usepackage{latexsym}
\usepackage{graphicx}
\usepackage{tabularx}
\usepackage{epsfig}
\usepackage{enumerate}

\usepackage[colorlinks,citecolor=red,linkcolor=blue]{hyperref}

\allowdisplaybreaks

\newtheorem{theorem}{Theorem}
\newtheorem{lemma}{Lemma}

\newtheorem{proposition}{Proposition}

\newtheorem{remark}{Remark}

\newcommand{\norm}[1]{\left\Vert#1\right\Vert}
\newcommand{\abs}[1]{\left\vert#1\right\vert}
\newcommand{\eps}{\varepsilon}
\newcommand{\bsa}{\boldsymbol{a}}
\newcommand{\bsk}{\boldsymbol{k}}
\newcommand{\bsl}{\boldsymbol{l}}

\newcommand{\bsx}{\boldsymbol{x}}

\newcommand{\bsh}{\boldsymbol{h}}

\newcommand{\bst}{\boldsymbol{t}}

\newcommand{\bsb}{\boldsymbol{b}}

\newcommand{\bsy}{\boldsymbol{y}}
\newcommand{\bszero}{\boldsymbol{0}}

\newcommand{\cO}{\mathcal{O}}

\newcommand{\cH}{\mathcal{H}}

\newcommand{\cA}{\mathcal{A}}
\newcommand{\cG}{\mathcal{G}}

\newcommand{\rd}{\,\mathrm{d}}
\newcommand{\NN}{\mathbb{N}}
\newcommand{\ZZ}{\mathbb{Z}}
\newcommand{\RR}{\mathbb{R}}

\newcommand{\APP}{{\rm APP}}
\newcommand{\INT}{{\rm INT}}
\newcommand{\lstd}{\Lambda^{\rm{std}}}
\newcommand{\lall}{\Lambda^{\rm{all}}}

\newcommand{\uu}{\mathfrak{u}}
\newcommand{\vv}{\mathfrak{v}}
\newcommand{\cHcos}{\cH(K_{s,\bsa,\bsb,\omega}^{\cos})}
\newcommand{\Kcos}{K_{s,\bsa,\bsb,\omega}^{\cos}}

\newcommand{\LONG}[1]{}

\begin{document}

\title{Integration and approximation in cosine spaces of smooth functions}

\author{Christian Irrgeher\thanks{C. Irrgeher is supported by the Austrian Science Fund (FWF): Project F5509-N26, which is a part of the Special Research Program "Quasi-Monte Carlo Methods: Theory and Applications".}\;,
Peter Kritzer\thanks{P. Kritzer is supported by the Austrian Science Fund (FWF): Project F5506-N26, which is a part of the Special Research Program "Quasi-Monte Carlo Methods: Theory and Applications".}\;,
Friedrich Pillichshammer\thanks{F. Pillichshammer is partially supported by the Austrian Science Fund (FWF): Project F5509-N26, which is a part of the Special Research Program "Quasi-Monte Carlo Methods: Theory and Applications".}}

\date{}

\maketitle

\begin{abstract}
We study multivariate integration and approximation for functions belonging to a weighted
reproducing kernel Hilbert space based on half-period cosine functions in the worst-case setting.
The weights in the norm of the function space depend on two sequences of real numbers and decay exponentially.
As a consequence the functions are infinitely often differentiable, and therefore
it is natural to expect exponential convergence of the worst-case error.
We give conditions on the weight sequences under which we have exponential convergence for the integration as well as the approximation problem.
Furthermore, we investigate the dependence of the errors on the dimension by considering various notions of tractability.
We prove sufficient and necessary conditions to achieve these tractability notions.
\end{abstract}

\noindent\textbf{Keywords:} numerical integration, function approximation, cosine space, worst-case error, exponential convergence, tractability

\noindent\textbf{2010 MSC:} 41A63, 41A25, 65C05, 65D30, 65Y20

\section{Introduction}\label{sec:introduction}

In this paper we study two instances of multivariate linear problems. We are interested in approximating linear operators $S_s: \cH_s\rightarrow \cG_s$, where $\cH_s$ is a certain Hilbert space of $s$-variate functions defined on $[0,1]^s$ and where $\cG_s$ is a normed space, namely:
\begin{itemize}
\item Numerical integration of functions $f\in\cH_s$: In this case, we have $\cG_s=\RR$ and
\begin{align*}
S_s(f)=\INT_s(f)=\int_{[0,1]^s} f(\bsx)\rd\bsx,
\end{align*}
or
\item $L_2$-approximation of functions $f\in \cH_s$: In this case, we have $\cG_s=L_2 ([0,1]^s)$ and
\begin{align*}
S_s(f)=\APP_s(f)=f.
\end{align*}
\end{itemize}
Without loss of generality, see, e.g., \cite{TWW88} or \cite[Section~4]{NW08}, we approximate $S_s$ by a linear algorithm $A_{n,s}$ using $n$ information evaluations which are given by linear functionals from the class $\Lambda\in\{\lall,\lstd\}$. More precisely, we approximate $S_s$ by algorithms of the form
\begin{align*}
A_{n,s}(f)=\sum_{j=1}^n \alpha_j L_j(f)\qquad\textnormal{for all}\quad f\in \cH_s,
\end{align*}
where $L_j\in \Lambda$ and $\alpha_j\in\cG_s$ for all $j=1,2,\dots,n$. For $\Lambda=\lall$ we have $L_j\in\cH_s^*$ whereas for $\Lambda=\lstd$ we have $L_j(f)=f(x_j)$ for all $f\in\cH_s$, and for some $x_j\in[0,1]^s$. Obviously, for multivariate integration only the class $\lstd$ makes sense. Furthermore, we remark that in this paper we consider only function spaces for which $\lstd\subset\lall$.

We measure the error of an algorithm $A_{n,s}$ in terms of the \textit{worst-case error}, which is defined as
\begin{align*}
e(A_{n,s},S_s):=\sup_{\substack{f \in \cH_s \\ \|f\|_{\cH_s} \le 1}} \norm{S_s(f)-A_{n,s}(f)}_{\cG_s},
\end{align*}
where $\norm{\cdot}_{\cH_s}$, $\norm{\cdot}_{\cG_s}$ denote the norms in $\cH_s$ and $\cG_s$, respectively. The \textit{$n$th minimal (worst-case) error} is given by
\begin{align*}
e(n,S_s):=\inf_{A_{n,s}} e(A_{n,s},S_s),
\end{align*}
where the infimum is taken over all admissible algorithms $A_{n,s}$. When we want to emphasize that the $n$th minimal error is taken with respect to algorithms using information from the class $\Lambda\in\{\lall,\lstd\}$, we write $e(n,S_s;\Lambda)$. For $n=0$, we consider algorithms that do not use information evaluations and therefore we use $A_{0,s}\equiv 0$. The error of $A_{0,s}$ is called the \textit{initial (worst-case) error} and is given by
\begin{align*}
e(0,S_s):=\sup_{\substack{f \in \cH_s \\ \|f\|_{\cH_s} \le 1}} \norm{S_s(f)}_{\cG_s}=\norm{S_s}.
\end{align*}

Since we will study a class of weighted reproducing kernel Hilbert spaces with exponentially decaying weights, which will be introduced in Section \ref{sec:cosine}, we are concerned with spaces $\cH_s$ of smooth functions. We remark that reproducing kernel Hilbert spaces of a similar flavor were previously considered in~\cite{DKPW13,DLPW11,IKLP15,IKPW15,IKPW15a,KPW14,KPW14a}. In this case it is natural to expect that, by using suitable algorithms, we should be able to obtain errors that converge to zero very quickly as $n$ increases, namely exponentially fast. By {\em exponential convergence} (EXP) for the worst-case error we mean that there exist a number $q\in(0,1)$ and functions $p,C,M:\NN\to (0,\infty)$ such that
\begin{align}\label{exrate}
e(n,S_s)\le C(s)\, q^{\,(n/M(s))^{\,p(s)}}\qquad\textnormal{for all} \quad s, n\in \NN.
\end{align}
If the function $p$ in \eqref{exrate} can be taken as a constant function, i.e., $p(s)=p>0$ for all $s\in\NN$, we say that we achieve \emph{uniform exponential convergence} (UEXP) for $e(n,S_s)$. Furthermore, we denote by $p^*(s)$ and $p^*$ the largest possible rates $p(s)$ and $p$ such that EXP and UEXP holds, respectively.

When studying algorithms $A_{n,s}$, we do not only want to control how their errors depend on $n$, but also how they  depend on the dimension $s$. This is of particular importance for high-dimensional problems. To this end, we define, for $\eps\in(0,1)$ and $s\in\NN$, the {\em information complexity} by
\begin{align*}
  n(\eps,S_s):=\min\left\{n\,:\,e(n,S_s)\leq\eps\,e(0,S_s) \right\}
\end{align*}
as the minimal number of information evaluations needed to reduce the initial error by a factor of $\eps$. Exponential convergence implies that asymptotically, with respect to $\eps$ tending to zero, we need $\mathcal{O}(\log^{1/p(s)} \eps^{-1})$ information evaluations to compute an $\eps$-approximation. However, it is not clear how long we have to wait to see this nice asymptotic behavior especially for large $s$. This, of course, depends on how $C(s)$, $M(s)$ and $p(s)$ depend on $s$, and this is the subject of tractability. Thus, we intend to study how the information complexity depends on $\log\eps^{-1}$ and $s$ by considering the following tractability notions, which were already considered in \cite{DKPW13,DLPW11,IKLP15, IKPW15, KPW14}. The nomenclature was introduced in \cite{KPW14a}. We say that we have {\em Exponential Convergence-Weak Tractability (EC-WT)} if
\begin{align*}
\lim_{s+\log\,\eps^{-1}\to\infty}\frac{\log\ n(\eps,S_s)}{s+\log\,\eps^{-1}}=0
\end{align*}
with the convention $\log\,0=0$, i.e., we rule out the cases for which $n(\eps,s)$ depends exponentially on $s$ and $\log\,\eps^{-1}$. If there exist numbers $c,\tau,\sigma>0$ such that
\begin{align}\label{eq:ECPT}
n(\varepsilon,S_s)\le c\,s^{\,\sigma}\,(1+\log\,\eps^{-1})^{\,\tau}\qquad\textnormal{for all}\quad s\in\NN,\ \eps\in(0,1),
\end{align}
we say that {\em Exponential Convergence-Polynomial Tractability (EC-PT)} holds. This means that the information complexity depends at most polynomially on $s$ and $\log\,\eps^{-1}$. If the upper bound is independent of the dimension $s$, i.e., \eqref{eq:ECPT} holds with $\sigma=0$, then we say that we have {\em Exponential Convergence-Strong Polynomial Tractability (EC-SPT)}. Furthermore, the infimum of all $\tau$ for which EC-SPT holds is called the {\em exponent of EC-SPT} and it is denoted by $\tau^*$.

For further general remarks on exponential convergence and these tractability notions we refer to \cite{DKPW13, DLPW11, IKLP15, IKPW15, KPW14, KPW14a}. Moreover, we remark  that in many papers tractability is studied for problems where we do not have exponential, but usually polynomial, error convergence. For this kind of problems, tractability is concerned with the question how the information complexity depends on $s$ and $\eps^{-1}$, rather than $\log \eps^{-1}$, (for a detailed survey of such results, we refer to~\cite{NW08,NW10,NW12}).

\subsection{Relations between multivariate integration and approximation}\label{secrelintapp}

It is well known that multivariate approximation using information from $\Lambda^\mathrm{std}$ is not easier than multivariate integration, see, e.g., \cite{NSW04}. More precisely, for any algorithm $A_{n,s}(f)=\sum_{k=1}^n \alpha_k f(\bsx_k)$ for multivariate approximation using the nodes $\bsx_1,\ldots,\bsx_n\in [0,1)^s$ and $\alpha_k\in L_2([0,1]^s)$, define $\beta_k:=\int_{[0,1]^s} \alpha_k (\bsx)\rd \bsx$ and the algorithm
\begin{align*}
A^{\rm int}_{n,s}(f)=\sum_{k=1}^n\beta_k\,f(\bsx_k)
\end{align*}
for multivariate integration. Then it can easily be checked that
\begin{align*}
e(A^{\rm int}_{n,s},\INT_s)\le e(A_{n,s},\APP_s).
\end{align*}
Since this holds for all algorithms $A_{n,s}$ we conclude that
\begin{align*}
e(n,\INT_s)\le e(n,\APP_s;\lstd).
\end{align*}
If the initial errors of integration and approximation are equal, these observations imply that for $\eps\in(0,1)$ and $s\in\NN$ we have
\begin{align}\label{eq:intapp2}
n(\eps,\INT_s) \le n(\eps,\APP_s;\lstd).
\end{align}

\subsection{The half-period cosine spaces with infinite smoothness}\label{sec:cosine}

We now introduce the function spaces for which we will consider multivariate integration and approximation. First, we choose two weight sequences of real positive numbers, $\bsa=\{a_j\}$ and $\bsb=\{b_j\}$ such that
\begin{align*}
0<a_*:=a_1\leq a_2\leq\ldots\quad\textnormal{ and }\qquad b_{\ast}:=\inf_j b_j>0.
\end{align*}
Moreover, we fix a parameter $\omega\in(0,1)$ and define for any $\bsk=(k_1,\ldots,k_s)\in\NN_0^s$,
\begin{align*}
\omega_{\bsk}:=\omega^{\sum_{j=1}^s a_j k_j^{b_j}}=\prod_{j=1}^{s}\omega^{a_j k_j^{b_j}}.
\end{align*}

The {\it half-period cosine space with infinite smoothness} is a reproducing kernel Hilbert space $\cHcos$ of (half-period) cosine series with reproducing kernel
\begin{align}\label{eq:kernel}
K_{s,\bsa,\bsb,\omega}^{\cos}(\bsx,\bsy):=\sum_{\bsk\in\NN_0^s}\omega_{\bsk}\,2^{\frac{\abs{\bsk}_*}{2}}2^{\frac{\abs{\bsk}_*}{2}} \prod_{j=1}^s \cos(\pi k_j x_j)\cos (\pi k_j y_j),\qquad\textnormal{for } \bsx, \bsy\in [0,1]^s,
\end{align}
where we define $|\bsk|_* := |\{j \in [s]\, : \, k_j \not=0\}|$ to be the number of nonzero components in $\bsk=(k_1,\ldots,k_s)\in\NN_0^s$, and $[s]:=\{1,\ldots,s\}$. The inner product in $\cHcos$ is given by
\begin{align*}
\langle f, g \rangle_{K_{s,\bsa,\bsb,\omega}^{\cos}} = \sum_{\bsk \in \mathbb{N}_0^s}\omega_{\bsk}^{-1} \widetilde{f}(\bsk) \, \widetilde{g}(\bsk),
\end{align*}
where the $\bsk$th cosine coefficient for a function $f:[0,1]^s \to\mathbb{R}$ is given by
\begin{align*}
\widetilde{f}(\bsk):=\int_{[0,1]^s} f(\bsx) \, 2^{\frac{|\bsk|_*}{2}} \prod_{j=1}^s \cos(\pi k_j x_j) \,\mathrm{d} \bsx.
\end{align*}
The corresponding norm is defined by $\|f\|_{K_{s,\bsa,\bsb,\omega}^{\cos}} = \sqrt{\langle f, f\rangle_{K_{s,\bsa,\bsb,\omega}^{\cos}}}$, in particular,
\begin{align}\label{normcossp}
  \|f\|_{K_{s,\bsa,\bsb,\omega}^{\cos}}^2=\sum_{\bsk \in \mathbb{N}_0^s} \frac{|\widetilde{f}(\bsk)|^2}{\omega_{\bsk}}=\sum_{\bsh \in \mathbb{Z}^s} 2^{-|\bsh|_*} \frac{|\widetilde{f}(|\bsh|)|^2}{\omega_{\bsh}},
\end{align}
where $|\bsh|=(|h_1|, \ldots, |h_s|)$. Furthermore, we can write $f\in\cHcos$ as
\begin{align*}
f(\bsx)=\sum_{\bsk\in\NN_0^s}\widetilde{f} (\bsk) \left(2^{\frac{\abs{\bsk}_*}{2}} \prod_{j=1}^s \cos (\pi k_j x_j) \right),
\end{align*}
and the sequence $\{e_{\bsk}\}_{\bsk\in\NN_0^s}$ with
\begin{align*}
e_{\bsk}(\bsx)=\left(2^{\frac{\abs{\bsk}_*}{2}} \prod_{j=1}^s \cos (\pi k_j x_j) \right)\omega_{\bsk}^{\frac{1}{2}}
\end{align*}
is an orthonormal basis of $\cHcos$.

From \eqref{normcossp} we see that the functions in $\cHcos$ are characterized by the decay rate of their cosine coefficients, which is regulated by the weights $\omega_{\bsk}$. Since $\{\omega_{\bsk}\}$ decreases exponentially as $\bsk$ grows, the cosine coefficients of the functions in $\cHcos$ also decrease exponentially fast, which influences the smoothness of the functions. Indeed, it can be shown that $f \in \cHcos$ belongs to the Gevrey class with index $1/b_*$ and so it follows that for the case $b_*\geq 1$ the functions of $\cHcos$ are analytic. Note that the half-period cosine space with infinite smoothness fits into the general setting which is discussed in \cite{IKPW15}. Furthermore we refer to the recent paper~\cite{DNP14} in which the integration problem is discussed in a half-period cosine space with weights of polynomial decay.

Finally, we remark that we can assume $a_*\geq 1$ without loss of generality, because we always can modify $\omega$ such that $a_*\geq 1$.

\section{Integration in half-period cosine spaces}

Let $\INT_s: \cHcos\to \RR$ with $\INT_s(f)=\int_{[0,1]^s}f(\bsx) \rd \bsx$. In order to approximate $\INT_s$ with respect to the absolute value $|\cdot|$ we use linear algorithms $A_{n,s}$, which are algorithms based on $n$ function evaluations of the form
\begin{align*}
A_{n,s}(f)=\sum_{k=1}^n \alpha_k f(\bsx_k)\quad \mbox{for}\ f\in\cHcos,
\end{align*}
where each $\alpha_k \in \RR$ and $\bsx_k \in [0,1]^s$ for $k=1,2,\ldots,n$.

The worst-case error of an algorithm $A_{n,s}$ is defined as
\begin{align*}
e(A_{n,s},\INT_s):=\sup_{\substack{f\in\cHcos\\ \norm{f}_{\Kcos\le 1}}}\vert\INT_s(f)- A_{n,s} (f)\vert.
\end{align*}
The $n$th minimal worst-case error is denoted by
\begin{align*}
e(n,\INT_s)=\inf_{A_{n,s}} e(A_{n,s},\INT_s),
\end{align*}
where the infimum is taken over all linear algorithms $A_{n,s}$ of the above given form. With \cite[Proposition~2.11]{DP10} it is easy to check the initial error is $e(0,\INT_s)=1$.

\subsection{Upper and lower error bounds}

An $n$-point midpoint rule is a linear integration rule of the form
\begin{align*}
T_n(f)=\frac{1}{n}\sum_{j=0}^{n-1} f\left(\frac{2 j+1}{2 n}\right).
\end{align*}

The following lemma on midpoint rules applied to cosine functions has a simple proof which we omit for the sake of brevity.

\begin{lemma}\label{le1}
We have
\begin{align*}
\left|\frac{1}{n}\sum_{j=0}^{n-1} \cos\left(\pi l \frac{2 j+1}{2n}\right)\right|= \begin{cases} 1 & \textnormal{if } l=2tn\textnormal{ for }\ t \in \ZZ,\\0 & \textnormal{otherwise}.
\end{cases}
\end{align*}
\end{lemma}

For integration in the multivariate case, we use the cartesian product of one-dimensional midpoint rules. Let $n_1,\ldots,n_s \in \NN$ and let $n=n_1 n_2\cdots n_s$. For $j=1,2,\ldots,s$ let $T_{n_j}^{(j)}(f)=\frac{1}{n_j}\sum_{i=0}^{n_j-1} f\left(\frac{2 i+1}{2 n_j}\right)$ be one-dimensional $n_j$-point midpoint rules. Then we apply the $s$-dimensional Cartesian product rule $T_{n,s}=T_{n_1} \otimes \cdots \otimes T_{n_s}$, i.e.,
\begin{align}\label{eq:midpoints}
T_{n,s}(f)=\frac{1}{n}\sum_{i_1=0}^{n_1-1}\ldots \sum_{i_s=0}^{n_s-1} f\left(\tfrac{2 i_1+1}{2 n_1},\ldots,\tfrac{2 i_s+1}{2 n_s}\right) \ \ \ \mbox{ for } \ f \in \cHcos.
\end{align}

\begin{proposition}\label{prop:wcints}
Let $T_{n,s}$ be the $s$-dimensional Cartesian product of $n_j$-point midpoint rules given by~\eqref{eq:midpoints} and let $n=n_1 n_2\cdots n_s$. Then we have
\begin{align*}
e^2(T_{n,s},\INT_s) \leq -1+\prod_{j=1}^s \left(1+\omega^{a_j (2 n_j)^{b_j}} C(a_*,b_*)\right)
\end{align*}
with $C(a_*,b_*):=2\omega^{-a_*} \sum_{l=1}^\infty \omega^{a_* l^{b_*}}<\infty$.
\end{proposition}

\begin{proof}
With \cite[Proposition 2.11]{DP10} and \eqref{eq:kernel} we get for the worst-case error of $T_{n,s}$ in $\cHcos$,
\begin{align*}
e^2(T_{n,s},\INT_s)&= \sum_{\bsl \in \NN_0^s \setminus \{\bszero\}} \omega_{\bsl} 2^{|\bsl|_0}\left(\frac{1}{n_1}\sum_{j_1=0}^{n_1-1} \ldots \frac{1}{n_s}\sum_{j_s=0}^{n_s-1}  \cos\left(\pi l_1 \frac{2 j_1+1}{2n_1}\right) \cdots \cos\left(\pi l_s \frac{2 j_s+1}{2n_s}\right) \right)^2\\
&= -1 +\prod_{i=1}^s\left(1+ 2 \sum_{l=1}^\infty \omega^{a_i l^{b_i}} \left(\frac{1}{n_i}\sum_{j=0}^{n_i-1} \cos\left(\pi l \frac{2 j+1}{2n_i}\right)\right)^2 \right).
\end{align*}
Now it follows from Lemma~\ref{le1} that
\begin{align*}
e^2(T_{n,s},\INT_s) &\leq -1 +\prod_{j=1}^s\left(1+ 2 \sum_{l=1}^\infty \omega^{a_j (2ln_j)^{b_j}}\right)\\
&\leq-1 +\prod_{j=1}^s\left(1+ 2\omega^{a_j (2n_j)^{b_j}} \omega^{-a_*}\sum_{l=1}^\infty \omega^{a_* l^{b_*}}\right).
\end{align*}
\end{proof}

The following proposition states a lower bound on the $n$th minimal worst-case error which will be useful to derive the necessary conditions in Theorem \ref{thm:intmain}.

\begin{proposition}\label{prop:lower}
Let $t_1,\ldots,t_s\in\NN$. Then the $n$th minimal worst-case error satisfies
\begin{align*}
e(n,\INT_s)\ge  \omega^{\sum_{j=1}^s a_j (2t_j)^{b_j}}\prod_{j=1}^s (5(1+ t_j))^{-1}
\end{align*}
for any $1\le n\leq \prod_{j=1}^s (1+t_j)$.
\end{proposition}
\begin{proof}
The lower bound can be shown in the same way as \cite[Lemma 1]{KPW14} or \cite[Theorem 2]{IKLP15}.
\end{proof}

\subsection{The main results for integration}\label{secmainint}

The following theorem gives necessary and sufficient conditions on the weight sequences $\bsa$ and $\bsb$ for EXP, UEXP, and the notions of EC-WT, EC-PT, and EC-SPT.

\begin{theorem}\label{thm:intmain}
Consider integration defined over $\cHcos$ with weight sequences $\bsa$ and $\bsb$.
\begin{enumerate}
\item\label{it:intexp}
EXP holds for all $\bsa$ and $\bsb$ considered, and $p^{*}(s)=\frac{1}{B(s)}$ with $B(s):=\sum_{j=1}^s\frac1{b_j}$.
\item\label{it:intuexp}
The following assertions are equivalent:
\begin{enumerate}[(i)]
\item The sequence $\{b_j^{-1}\}_{j \ge 1}$ is summable, i.e., $B:=\sum_{j=1}^\infty\frac1{b_j}<\infty$;
\item we have UEXP;
\item we have EC-PT;
\item we have EC-SPT.
\end{enumerate}
If one of the assertions holds then $p^*=1/B$ and the exponent $\tau^{\ast}$ of EC-SPT is $B$.
\item\label{it:wtnec} EC-WT implies that $\lim_{j \rightarrow \infty} a_j 2^{b_j} =\infty$.
\item\label{it:wtsuff} A sufficient condition for EC-WT is that $\lim_{j \rightarrow \infty} a_j=\infty$.
\end{enumerate}
\end{theorem}

\begin{proof}
{\bf Proof of Point \ref{it:intexp}:} We first show that we always have EXP. For $\eps\in(0,1)$, define
\begin{align*}
m=\max_{j=1,2,\dots,s}\,\left\lceil\left(\frac{1}{a_j}\,\frac{\log\left(C(a_*,b_*)\frac{s}{\log(1+\eps^2)}\right)}{\log\, \omega^{-1}}\right)^{B(s)}\,\right\rceil
\end{align*}
with $C(a_*,b_*)=2\omega^{-a_*}\sum_{k=1}^\infty \omega^{a_* k^{b_*}}$. Now let $n_1,n_2,\ldots,n_s$ be given by
\begin{align*}
n_j:=\left\lfloor m^{1/(B(s)  b_j)}\right\rfloor\qquad\textnormal{for}\quad  j=1,2,\ldots,s
\end{align*}
and set $n=n_1n_2\cdots n_s$. Since $\lfloor x\rfloor\ge x/2$ for all $x\ge1$, we have
\begin{align*}
a_j(2 n_j)^{b_j}\ge a_j \,m^{1/B(s)}
\end{align*}
for every $j=1,2,\ldots,s$. Then we obtain with Proposition \ref{prop:wcints},
\begin{align*}
e^2(T_{n,s},\INT_s) \le -1+\prod_{j=1}^s\left(1+\omega^{a_j m^{1/B(s)}} C(a_*,b_*)\right).
\end{align*}
From the definition of $m$ we have for all $j=1,2,\dots,s$,
\begin{align*}
\omega^{a_j m^{1/B(s)}}C(a_*,b_*) \le \frac{\log(1+\eps^2)}{s}.
\end{align*}
This proves
\begin{align}\label{eq:eexp}
e(T_{n,s},\INT_s)\le \left(-1+\left(1+\frac{\log(1+\eps^2)}{s}\right)^s\,\right)^{1/2} \le \left(-1+\exp(\log(1+\eps^2))\right)^{1/2}=\eps.
\end{align}
Furthermore, note that
\begin{align}\label{eq:nexp}
n=\prod_{j=1}^s n_j=\prod_{j=1}^s \left\lfloor m^{1/(B(s) b_j)} \right\rfloor \le m^{\frac{1}{B(s)}\sum_{j=1}^s 1/b_j}= m =\cO\left(\log^{\,B(s)}\left(1+\eps^{-1}\right)\right),
\end{align}
with the factor in the $\cO$ notation independent of $\eps^{-1}$ but dependent on $s$. From \eqref{eq:eexp} and \eqref{eq:nexp} it follows directly that we always have EXP with $p^*(s)\geq 1/B(s)$.

It remains to show that we have indeed $p^*(s)=1/B(s)$ for the largest possible rate of EXP. To this end, assume that EXP holds. Let $\bst=(t_1,\ldots,t_s)\in\NN^s$ and choose $n=-1+\prod_{j=1}^s (1+t_j)$. Then Proposition~\ref{prop:lower} implies
\begin{align*}
C(s) q^{(n/M(s))^{p(s)}}\ge  \omega^{\sum_{j=1}^s a_j (2t_j)^{b_j}}\prod_{j=1}^s (5(1+ t_j))^{-1},
\end{align*}
which implies
\begin{align}
&\frac{\sum_{j=1}^s a_j (2t_j)^{b_j}}{\prod_{j=1}^s (1+t_j)^{p(s)}}+\frac{\log C(s)+\sum_{j=1}^s \log (5(1+t_j))}{\prod_{j=1}^s (1+t_j)^{p(s)}\log\omega^{-1}} \geq\nonumber\\
&\qquad\qquad\qquad\qquad\qquad\ge \left(1-\frac{1}{\prod_{j=1}^s (1+t_j)}\right)^{p(s)} \frac{\log q^{-1}}{\log \omega^{-1}} M(s)^{-p(s)}.\label{eq:B1}
\end{align}
For fixed $s$, when $\norm{\bst}_{s,\infty}=\max_{1\le j\le s} t_j\to\infty$, then the second term of the left hand side of \eqref{eq:B1} goes to zero, and it follows that
\begin{align*}
\liminf_{\norm{\bst}_{s,\infty}\to\infty} \frac{\sum_{j=1}^s a_j (2t_j)^{b_j}}{\prod_{j=1}^s (1+t_j)^{p(s)}}\ge \frac{\log q^{-1}}{\log\omega^{-1}}M(s)^{-p(s)}>0.
\end{align*}
For a positive number $t$ take now
\begin{align*}
t_j:=\left\lceil t^{1/b_j}\right\rceil\ \mbox{for all } j=1,2,\ldots,s.
\end{align*}
Clearly, $\lim_{t\to\infty} \left\lceil t^{1/b_j}\right\rceil /t^{1/b_j}=1$. Then we have
\begin{align}
\frac{\sum_{j=1}^s a_j (2t_j)^{b_j}}{\prod_{j=1}^s (1+t_j)^{p(s)}} 
=t^{1-p(s)\sum_{j=1}^s b_j^{-1}}\frac{\sum_{j=1}^s a_j (2\left\lceil t^{1/b_j}\right\rceil/t^{1/b_j})^{b_j}}{\prod_{j=1}^s (\left\lceil t^{1/b_j}\right\rceil /t^{1/b_j}+t^{-1/b_j})^{p(s)}}.\label{eq:B2}
\end{align}
Now if $t\to\infty$, then
\begin{align*}
\frac{\sum_{j=1}^s a_j (2\left\lceil t^{1/b_j}\right\rceil/t^{1/b_j})^{b_j}}{\prod_{j=1}^s (\left\lceil t^{1/b_j}\right\rceil /t^{1/b_j}+t^{-1/b_j})^{p(s)}}\longrightarrow \sum_{j=1}^s a_j 2^{b_j}.
\end{align*}
However, the expression \eqref{eq:B2} must be positive when $t$ tends to infinity, so we must have $p(s)\sum_{j=1}^s b_j^{-1}\le 1$ or, equivalently, $p(s)\leq 1/B(s)$. Thus, altogether $p(s)=1/B(s)$.

\vspace{0.25cm}
\noindent {\bf Proof of Point \ref{it:intuexp}:} Obviously, EC-SPT implies EC-PT. In \cite{KPW14} it is shown that EC-PT implies UEXP. Now we assume that UEXP holds, then we can show in exactly the same way as above that
\begin{align*}
\liminf_{\norm{\bst}_{s,\infty}\to\infty} \frac{\sum_{j=1}^s a_j (2t_j)^{b_j}}{\prod_{j=1}^s (1+t_j)^{p}}\ge \frac{\log q^{-1}}{\log\omega^{-1}}M(s)^{-p}>0.
\end{align*}
From this, it is then shown analogously to the proof of the first point that $p\sum_{j=1}^s b_j^{-1}\le 1$. Since this holds for all $s$, we conclude, for $s$ tending to infinity, that $p\sum_{j=1}^\infty 1/b_j=pB \leq 1$. Furthermore, if we assume that $B<\infty$, and if we replace $B(s)$ by $B$ in the proof of Point \ref{it:intexp}, we see that UEXP holds with $p\geq 1/B$. Thus we have that $p^*=1/B$.

It remains to show that $B<\infty$ implies EC-SPT. So assume that $B=\sum_{j=1}^\infty 1/b_j<\infty$ and let $n_1,\ldots,n_s$ be given by
\begin{align*}
n_j=\left\lceil\left(\frac{\log\left(C(a_*,b_*)\,\frac{\pi^2}{6}\,\frac{j^2}{\log(1+\eps^2)}\right)}{a_j 2^{b_j}\,\log \omega^{-1}}\right)^{1/b_j}\right\rceil
\end{align*}
with $C(a_*,b_*)=2\omega^{-a_*}\sum_{k=1}^\infty \omega^{a_* k^{b_*}}$. Note that $n_j$ is defined such that
\begin{align*}
\omega^{a_j(2 n_j)^{b_j}}C(a_*,b_*)\leq\frac{6}{\pi^2}\ \frac{\log(1+\eps^2)}{j^2}.
\end{align*}
Therefore
\begin{align*}
e^2(T_{n,s},\INT_s)\leq& -1+\prod_{j=1}^s\left(1+\frac{6}{\pi^2}\,\frac{\log(1+\eps^2)}{j^2}\right)\\
\leq&-1 +\exp\left(\frac{6}{\pi^2}\,\log(1+\eps^2)\ \sum_{j=1}^sj^{-2}\right)\\
\leq&-1+\exp\left(\log(1+\eps^2)\right)=\eps^2.
\end{align*}

We now estimate $n_j$ and then $n=\prod_{j=1}^s n_j$. Clearly, $n_j\ge1$ for all $j\in\NN$. We prove that $n_j=1$ for large $j$. Indeed, $n_j=1$ if
\begin{align}\label{lasteq}
a_j 2^{b_j}\log \omega^{-1}\ge \log\left(C(a_*,b_*)\frac{\pi^2}6\,\frac{j^2}{\log(1+\eps^2)}\right).
\end{align}

Let $\delta>0$. From $\sum_j \frac{1}{b_j}< \infty$ it follows with the Cauchy condensation test that also $\sum_j \frac{2^j}{b_{2^j}}< \infty$ and hence $\lim_j 2^j/b_{2^j} =0$. Hence, we find that $b_{2^j}^{-1} \le \frac{\delta}{2^{j+1}}$ for $j$ large enough.
For large enough $k$ with $2^j \leq k \leq 2^{j+1}$ we then obtain
\begin{align*}
\frac{1}{b_k} \leq \frac{1}{b_{2^{j}}} \leq \frac{\delta}{2^{j+1}} \le \frac{\delta}{k}
\end{align*}
or, equivalently, $2^{b_k} \ge 2^{k/\delta}$. Hence, there exists a positive $\beta_1$ such that
\begin{align*}
a_k 2^{b_k} \geq \beta_1 2^{k/\delta}\qquad\textnormal{for all}\quad k \in \NN.
\end{align*}
Then the inequality \eqref{lasteq} holds for all $j\ge j^*$, where $j^*$ is the smallest positive integer for which
\begin{align*}
j^*\ge \frac{\delta}{\log\,2}\, \log\left(\frac1{\beta_1\log \omega^{-1}}\ \log\left(C(a_*,b_*)\frac{\pi^2}6\, \frac{[j^*]^2}{\log(1+\eps^2)}\right)\right).
\end{align*}
Clearly,
\begin{align*}
j^*= \frac{\delta}{\log\,2}\log\ \log\ \eps^{-1}+\cO(1) \qquad\textnormal{as}\qquad\eps\to0.
\end{align*}
Without loss of generality we can restrict ourselves to $\eps\le {\rm e}^{-{\rm e}}$, where ${\rm e}=\exp(1)$, so that $\log\,\log\,\eps^{-1}\ge1$. Then there exists a number $C_0\ge1$, independent of $\eps$ and $s$,  such that
\begin{align*}
n_j=1\quad\textnormal{for all}\quad j> \left\lfloor C_0+\frac{\delta}{\log\,2}\,\log\ \log\ \eps^{-1}\right\rfloor.
\end{align*}
We now estimate $n_j$ for $j\le \left\lfloor C_0+\frac{\delta}{\log\,2}\, \log\,\log\,\eps^{-1}\right\rfloor$. Note that
\begin{align*}
\log\left(\frac{2}{1-\omega}\frac{\pi^2}6\ \frac{j^2}{\log(1+\eps^2)}\right)=\log\left(C(a_*,b_*)\frac{\pi^2}{6\,\log(1+\eps^2)}\right)\ +\ \log(j^2).
\end{align*}
Then $a_j 2^{b_j}\ge 2^{j/\delta}$ also implies that
\begin{align*}
K(\bsa,\bsb):= \sup_{j\in\NN}\frac{\log(j^2)}{a_j 2^{b_j}}<\infty.
\end{align*}
Furthermore, there exists a number $C_1\geq1$, independent of $\eps$ and $s$ such that
\begin{align*}
\log\left(C(a_*,b_*)\frac{\pi^2}{6\,\log(1+\eps^2)}\right)\le C_1+2\log\frac1\eps\qquad\textnormal{for all}\quad\eps\in(0,1).
\end{align*}
This yields
\begin{align*}
n_j\leq1+\left(\frac{K(\bsa,\bsb)+C_1+2\log \eps^{-1}}{\log\omega^{-1}}\right)^{1/b_j}\qquad\textnormal{for all}\quad j\leq \left\lfloor C_0+\frac{\delta}{\log\,2}\,\log\,\log \frac{1}{\eps}\right\rfloor.
\end{align*}
Let
\begin{align*}
k=\min\left(s,\left\lfloor C_0+\frac{\delta}{\log\,2} \,\log\,\log\frac1\eps\right\rfloor\right).
\end{align*}
Then for $C=\max\left(K(\bsa,\bsb)+C_1,-2{\rm e}+\log \omega^{-1}\right)$ we have
\begin{align*}
\max\left(1,\frac{K(\bsa,\bsb)+C_1+2\log\eps^{-1}}{\log \omega^{-1}}\right) \le \frac{C+2\log\eps^{-1}}{\log \omega^{-1}}
\end{align*}
and
\begin{align*}
n=&\prod_{j=1}^sn_j=\prod_{j=1}^kn_j\le\prod_{j=1}^k \left(1+\left(\frac{C+2\log\eps^{-1}} {\log \omega^{-1}}\right)^{1/b_j}\right)\\
=&\left(\frac{C+2\log\eps^{-1}} {\log \omega^{-1}} \right)^{\sum_{j=1}^k1/b_j}\,\prod_{j=1}^k\left(1+\left(\frac{\log \omega^{-1}}{C+2\log\eps^{-1}}\right)^{1/b_j}\right)\\
\le&\left(\frac{C+2\log\eps^{-1}} {\log \omega^{-1}}\right)^{B}\,2^{k}.
\end{align*}
Note that
\begin{align*}
2^k\le 2^{C_0}\,\exp\left(\delta \log\,\log\frac{1}{\eps}\right)=2^{C_0}\,\log^{\delta}\frac1\eps.
\end{align*}
Therefore for any positive $\delta$ there is a positive number $C_\delta$ independent of $\eps^{-1}$ and $s$ such that
\begin{align*}
n(\eps,s)\leq n\leq C_\delta\,\log^{B+\delta}\left(1+\frac1\eps\right)\qquad \textnormal{for all}\qquad\eps\in(0,1),\ s\in\NN.
\end{align*}
This means that we have EC-SPT with $\tau^\ast$ at most $B$.

\vspace{0.25cm}
\noindent {\bf Proof of Point \ref{it:wtnec}:} Assume that EC-WT holds and that $(a_j 2^{b_j})_{j \ge 0}$ is bounded, say $a_j 2^{b_j}\le A<\infty$ for all $j \in \NN$. From setting $t_1=t_2=\cdots=1$ in Proposition~\ref{prop:lower} it follows that for all $n<2^s$ we have
\begin{align*}
e(n,\INT_s)\geq 10^{-s} \,\omega^{\sum_{j=1}^s a_j2^{b_j}}\geq 10^{-s}\,\omega^{A s}= \eta^s,
\end{align*}
where $\eta:= \omega^{A}/10 \in (0,1)$. Hence, for $\eps=\eta^s/2$ we have $e(n,\INT_s)>\eps$ for all $n<2^s$. This implies that $n(\eps,\INT_s)\geq2^s$ and
\begin{align*}
\frac{\log n(\eps,\INT_s)}{s+\log \eps^{-1}}\geq\frac{s \log 2}{s+\log 2 +s \log \eta^{-1}} \longrightarrow  \frac{\log 2}{1+\log \eta^{-1}} >0\quad \textnormal{as } s \rightarrow \infty,
\end{align*}
but this contradicts EC-WT. Therefore, it must hold that $\lim_{j\to\infty}a_j 2^{b_j}=\infty$.

\vspace{0.25cm}
\noindent {\bf Proof of Point \ref{it:wtsuff}:} By Point \ref{it:appwt} of Theorem~\ref{thm:appmain}, the condition $\lim_{j\rightarrow \infty} a_j =\infty$ implies EC-WT for the approximation problem with $\Lambda^{{\rm std}}$, which, by \eqref{eq:intapp2}, also implies EC-WT for the integration problem.
\end{proof}

\section{$L_2$-approximation in half-period cosine spaces}

Let $\APP_s: \cHcos\to L_2 ([0,1]^s)$ with $\APP_s(f)=f$. In order to approximate $\APP_s$ with respect to the norm in $L_2 ([0,1]^s)$ we use linear algorithms $A_{n,s}$, which are algorithms based on $n$ information evaluations of the form
\begin{align*}
A_{n,s}(f)=\sum_{k=1}^n \alpha_k L_k (f)\quad \mbox{for}\ f\in\cHcos,
\end{align*}
where each $\alpha_k$ is a function from $L_2 ([0,1]^s)$ and each $L_k$ is a continuous linear functional defined on $\cHcos$ from a permissible class $\Lambda$ of information.

The worst-case error of an algorithm $A_{n,s}$ is defined as
\begin{align*}
e(A_{n,s},\APP_s):=\sup_{\substack{f\in \cHcos\\ \norm{f}_{\Kcos}\le 1}} \norm{f- A_{n,s} (f)}_{L_2 ([0,1]^s)}.
\end{align*}
The $n$th minimal worst-case error for the information class $\Lambda\in\{\Lambda^{\mathrm{std}},\Lambda^{\mathrm{all}}\}$ is denoted by
\begin{align*}
e(n,\APP_s;\Lambda)=\inf_{A_{n,s}} e(A_{n,s},\APP_s),
\end{align*}
where the infimum is taken over all linear algorithms $A_{n,s}$ using information from $\Lambda$. For $n=0$, we simply approximate $f$ by zero, and the initial error is $e(0,\APP_s;\Lambda)=1$.

\subsection{Some auxiliary results}\label{applambdastd}

We proceed in a similar way to \cite{DKPW13} and \cite{IKPW15a}. For $M>1$ we define the set $\cA(s,M)=\{\bsh\in\NN_0^s: \omega_{\bsh}^{-1}<M\}$ and we will study approximating $f\in\cHcos$ by algorithms of the form
\begin{align}\label{eqdefAnsM}
A_{n,s,M}(f)(\bsx)= \sum_{\bsh \in \cA(s,M)} \left(\frac{1}{n}\sum_{k=1}^{n} f(\bsx_k) 2^{\frac{|\bsh|_*}{2}} \prod_{j=1}^s \cos(\pi h_j x_{k,j}) \right) 2^{\frac{|\bsh|_*}{2}} \prod_{j=1}^s \cos(\pi h_j x_j),
\end{align}
where $\bsx=(x_1,\ldots,x_s)\in[0,1]^s$. The choice of $M$ and $\bsx_k=(x_{k,1},\ldots,x_{k,s}) \in [0,1)^s$ will be given later.

We first show upper bounds on the worst-case error of $A_{n,s,M}$. The following analysis is similar to that in \cite{IKPW15a, KSW06}. Using Parseval's identity we obtain
\begin{align}
&\|f-A_{n,s,M}(f)\|_{L_2([0,1]^s)}^2=\nonumber\\
&\qquad\quad=\sum_{\bsh \notin\cA(s,M)} |\widetilde{f}(\bsh)|^2+\sum_{\bsh \in \cA(s,M)}\left|\widetilde{f}(\bsh)-\frac{1}{n}\sum_{k=1}^n f(\bsx_k) 2^{\frac{|\bsh|_*}{2}} \prod_{j=1}^s \cos(\pi h_j x_{k,j})\right|^2.\label{approx_err}
\end{align}
The first term can be estimated by
\begin{align}\label{approx_err1}
\sum_{\bsh \notin\cA(s,M)}|\widetilde{f}(\bsh)|^2 = \sum_{\bsh \notin\cA(s,M)}|\widetilde{f}(\bsh)|^2 \omega_{\bsh}\,\omega_{\bsh}^{-1}\leq\frac{1}{M} \|f\|_{\Kcos}^2.
\end{align}

We now consider the second term in \eqref{approx_err}. Let $\bsh\in\NN_0^s$. For $f\in\cHcos$ we define $f_{\bsh}(\bsx):=f(\bsx) 2^{|\bsh|_*/2}\prod_{j=1}^s \cos(\pi h_j x_j)$. So we get
\begin{align*}
\left|\widetilde{f}(\bsh)-\frac{1}{n}\sum_{k=1}^n f(\bsx_k) 2^{|\bsh|_*/2} \prod_{j=1}^s \cos(\pi h_j x_{k,j})\right|^2=\left|\int_{[0,1]^s}f_{\bsh}(\bsx) \rd \bsx-\frac{1}{n}\sum_{k=1}^{n} f_{\bsh}(\bsx_k)\right|^2.
\end{align*}

As for the integration problem, we choose the points $\bsx_1,\ldots,\bsx_n$ used in the algorithm $A_{n,s,M}$ according to a centered regular grid $\cG_{n,s}$ with different mesh-sizes $n_1,\ldots,n_s \in \NN$ for successive variables, i.e.,
\begin{align*}
\cG_{n,s}=\left\{\left(\frac{2 k_1+1}{2 n_1},\ldots, \frac{2 k_s+1}{2 n_s}\right)\,:\ k_j=0,1,\ldots,n_j -1 \textnormal{ for all } j=1,2,\ldots,s\right\},
\end{align*}
where $n=\prod_{j=1}^s n_j$ is the cardinality of $\cG_{n,s}$. By $\cG_{n,s}^{\bot}$ we denote the set
\begin{align*}
\cG_{n,s}^{\bot}=\{\bsh \in \NN_0^s \, : \, h_j \equiv 0 \pmod{2 n_j} \ \mbox{ for all } \ j=1,2,\ldots,s\}.
\end{align*}
Since $f_{\bsh}$ can be represented by a pointwise convergent cosine series, we have that
\begin{align*}
\left|\int_{[0,1]^s}f_{\bsh}(\bsx) \rd \bsx-\frac{1}{n}\sum_{k=1}^{n} f_{\bsh}(\bsx_k)\right|&\leq \left|\sum_{\bsl \in \NN_0^s \setminus \{\bszero\}} \widetilde{f_{\bsh}}(\bsl) 2^{\frac{|\bsl|_*}{2}}  \prod_{j=1}^s \left( \frac{1}{n_j} \sum_{k_j=1}^{n_j}\cos\left(\pi l_j \frac{2 k_j+1}{2 n_j}\right)\right) \right| \\
&\leq \sum_{\bsl \in \cG_{n,s}^{\bot} \setminus \{\bszero\}}| \widetilde{f_{\bsh}}(\bsl)| \ 2^{\frac{|\bsl|_*}{2}}.
\end{align*}
Using the fact that
\begin{align*}
\cos(\pi h t) \cos(\pi l t)= \frac{\cos(\pi(h+l)t)+\cos(\pi(h-l)t)}{2}
\end{align*}
we obtain
\begin{align*}
\widetilde{f_{\bsh}}(\bsl)&=\frac{1}{2^s} \int_{[0,1]^s} f(\bst)2^{\frac{|\bsh|_*}{2}}2^{\frac{|\bsl|_*}{2}} \prod_{j=1}^s ( \cos(\pi(h_j+l_j)t_j)+\cos(\pi(h_j-l_j)t_j) ) \rd \bst\\
&=\frac{1}{2^s}\sum_{\uu \subseteq [s]} \int_{[0,1]^s} f(\bst) 2^{\frac{|\bsh|_*}{2}}2^{\frac{|\bsl|_*}{2}}\prod_{j \in \uu} \cos(\pi(h_j+l_j)t_j) \prod_{j \not\in \uu} \cos(\pi(h_j-l_j)t_j) \rd \bst \\
&=\frac{1}{2^s} \sum_{\uu \subseteq [s]} \widetilde{f}(\bsh (\pm)_{\uu} \bsl),
\end{align*}
where $\bsh (\pm)_{\uu} \bsl =(y_1,y_2,\ldots,y_s)$ with
\begin{align*}
y_j=\begin{cases}h_j+l_j & \textnormal{if } j \in \uu,\\|h_j-l_j| & \textnormal{if } j \not\in \uu.\end{cases}
\end{align*}
Thus we obtain
\begin{align*}
\left|\int_{[0,1]^s}f_{\bsh}(\bsx)\rd\bsx-\frac{1}{n}\sum_{k=1}^{n} f_{\bsh}(\bsx_k)\right|^2 &\leq\left(\sum_{\bsl\in\cG_{n,s}^{\bot}\setminus\{\bszero\}}2^{\frac{|\bsl|_*}{2}}\left|\frac{1}{2^s}\sum_{\uu\subseteq [s]}\widetilde{f}(\bsh (\pm)_{\uu}\bsl)\right|\right)^2\\
\le & \left(\frac{1}{2^s}\sum_{\uu\subseteq [s]}\sum_{\bsl\in\cG_{n,s}^{\bot}\setminus\{\bszero\}}2^{\frac{|\bsl|_*}{2}}\left|\widetilde{f}(\bsh(\pm)_{\uu} \bsl)\right|\omega_{\bsh(\pm)_{\uu}\bsl}^{-1/2}\omega_{\bsh(\pm)_{\uu}\bsl}^{1/2}\right)^2.
\end{align*}
Using the Cauchy-Schwarz inequality we obtain
\begin{align}
&\left|\int_{[0,1]^s} f_{\bsh}(\bsx)\rd\bsx-\frac{1}{n}\sum_{k=1}^{n} f_{\bsh}(\bsx_k) \right|^2\nonumber \\
& \qquad\qquad\leq\left(\frac{1}{2^s} \sum_{\uu \subseteq [s]} \sum_{\bsl \in \cG_{n,s}^{\bot} \setminus \{\bszero\}}  \left|\widetilde{f}(\bsh (\pm)_{\uu} \bsl)\right|^2 \omega_{\bsh (\pm)_{\uu} \bsl}^{-1}\right) \left(\frac{1}{2^s} \sum_{\uu \subseteq [s]} \sum_{\bsl \in \cG_{n,s}^{\bot} \setminus \{\bszero\}} 2^{|\bsl|_*} \omega_{\bsh (\pm)_{\uu} \bsl}\right).\label{errfhab2}
\end{align}
We know that the coefficient $\widetilde{f}(\bsh (\pm)_{\uu} \bsl)$ can occur at most $2^{s-\vert\uu\vert}$ times in \eqref{errfhab2} and therefore we get
\begin{align*}
\frac{1}{2^s} \sum_{\uu \subseteq [s]} \sum_{\bsl \in \cG_{n,s}^{\bot} \setminus \{\bszero\}}\left|\widetilde{f}(\bsh (\pm)_{\uu} \bsl)\right|^2 \omega_{\bsh (\pm)_{\uu} \bsl}^{-1}&\le \frac{1}{2^s}\sum_{\uu \subseteq [s]} \sum_{\bsk \in \NN_0^s} 2^{s-\vert \uu\vert} \left|\widetilde{f}(\bsk)\right|^2 \omega_{\bsk}^{-1}\\
&= \frac{1}{2^s} \sum_{\uu \subseteq [s]} 2^{s-\vert \uu\vert} \|f\|^2_{\Kcos}\\
&= \|f\|^2_{\Kcos}\left(\frac{3}{2}\right)^s.
\end{align*}
Altogether we achieve
\begin{align}\label{errfhab}
\left|\int_{[0,1]^s} f_{\bsh}(\bsx)\rd\bsx-\frac{1}{n}\sum_{k=1}^{n} f_{\bsh}(\bsx_k) \right|^2 \le \|f\|^2_{\Kcos} \left(\frac{3}{2}\right)^s\left(\frac{1}{2^s}
\sum_{\uu \subseteq [s]} \sum_{\bsl \in \cG_{n,s}^{\bot} \setminus \{\bszero\}} 2^{|\bsl|_*} \omega_{\bsh (\pm)_{\uu} \bsl}\right).
\end{align}
For any $h,l\in\NN_0$ we have
\begin{align*}
\abs{l}^b=|l \pm h \mp h|^b \le 2^{b-1}\left(\abs{l\pm h}^b + \abs{h}^b\right)
\end{align*}
if $b\geq 1$ and
\begin{align*}
\abs{l}^b=|l \pm h \mp h|^b \le\left(\abs{l\pm h}^b + \abs{h}^b\right)
\end{align*}
if $b<1$. Thus we obtain
\begin{align*}
\abs{l\pm h}^b\geq \left(\abs{l/2}^b-\abs{h}^b \right)
\end{align*}
for any $h,l\in\NN_0$ and for any $b>0$. For $\bsh\in\cA(s,M)$ and $\uu \subseteq [s]$ this implies
\begin{align}\label{bds_omegakl}
\omega_{\bsl (\pm)_{\uu} \bsh}=\omega^{\sum_{j\in u} a_j|l_j+h_j|^{b_j} + \sum_{j\not \in u} a_j|l_j-h_j|^{b_j} }\le\omega^{\sum_{j=1}^s a_j\abs{l_j/2}^{b_j}}  \omega^{-\sum_{j=1}^s a_j\abs{h_j}^{b_j}}\le \omega^{\sum_{j=1}^s  a_j\abs{l_j/2}^{b_j}} M.
\end{align}
Hence
\begin{align}\label{bdappint1}
\left|\int_{[0,1]^s} f_{\bsh}(\bsx)\rd\bsx-\frac{1}{n}\sum_{k=1}^{n} f_{\bsh}(\bsx_k) \right|^2\le\|f\|^2_{\Kcos}\;\left(\frac{3}{2}\right)^s M\ \sum_{\bsl\in\cG_{n,s}^{\bot}\setminus\{\bszero\}} 2^{|\bsl|_*}\omega^{\sum_{j=1}^s a_j\abs{l_j/2}^{b_j}}.
\end{align}

Therefore, and using \eqref{approx_err}, \eqref{approx_err1}, and \eqref{bdappint1} for any $f\in H(K_{s,\bsa,\bsb,\omega}^{\cos})$ with $\|f\|_{\cHcos} \leq 1$, we obtain
\begin{align}\label{gen_approx_bd}
\norm{f-A_{n,s,M}(f)}_{L_2([0,1]^s)}^2 & \le \frac{1}{M} + M \left(\frac{3}{2}\right)^s |\cA(s,M)|\, 2^sF_n,
\end{align}
where
\begin{align}\label{eq:deffn}
F_n := \sum_{\bsl \in \cG_{n,s}^{\bot} \setminus \{\bszero\}}\omega^{\sum_{j=1}^s a_j (l_j/2)^{b_j}}.
\end{align}
Furthermore,
\begin{align*}
|\cA(s,M)|\leq\prod_{j=1}^s\left(1+\left(\frac{\log M}{a_j \log \omega^{-1}}\right)^{1/b_j}\right)&\le \prod_{j=1}^s\left(1+\left(\frac{\log M}{ \log \omega^{-1}}\right)^{1/b_j}\right)\\
&\le \prod_{j=1}^s\left(1+ \log^{-1/b_j}\omega^{-1}\right)\prod_{j=1}^s\left(1+\log^{1/b_j}M\right).
\end{align*}
Since $M$ is assumed to be at least 1, we can bound  $1+\log^{1/b_j}M\le 2M^{1/b_j}$, and obtain
\begin{align*}
|\cA(s,M)|\le 2^s M^{B(s)} \prod_{j=1}^s \left(1+ \log^{-1/b_j} \omega^{-1}\right),
\end{align*}
where, as in the previous sections, $B(s):=\sum_{j=1}^s b_j^{-1}$. Plugging this into \eqref{gen_approx_bd}, we obtain
\begin{align}\label{eqboundMapp2}
[e(A_{n,s,M},\APP_s;\lstd)]^2\le\frac{1}{M}+ M^{B(s)+1} D(s,\omega,\bsb) F_n,
\end{align}
where
\begin{align*}
D(s,\omega,\bsb):=6^s \prod_{j=1}^s\left(1+  \log^{-1/b_j}\omega^{-1}\right).
\end{align*}

\subsection{The main results for $L_2$-approximation}\label{secmainapp}

The following theorem gives necessary and sufficient conditions on the weight sequences $\bsa$ and $\bsb$ for (uniform) exponential convergence, and the notions of EC-WT, EC-PT, and EC-SPT.

\begin{theorem}\label{thm:appmain}
Consider $L_2$-approximation defined over $\cHcos$ with weight sequences $\bsa$ and $\bsb$. The following results hold for both classes $\lall$ and $\lstd$.
\begin{enumerate}
\item\label{it:appexp}
EXP holds for arbitrary $\bsa$ and $\bsb$ and $p^{*}(s)=1/B(s)$ with $B(s):=\sum_{j=1}^s\frac{1}{b_j}$.
\item\label{it:appuexp}
UEXP holds iff $\bsa$ is an arbitrary sequence and $\bsb$ such that $B:=\sum_{j=1}^\infty\frac1{b_j}<\infty$.
If so then $p^*=1/B$.
\item\label{it:appwt}
We have EC-WT iff $\lim_{j\to\infty}a_j=\infty$.
\item \label{it:appspt}
EC-PT holds iff EC-SPT holds iff
\begin{align*}
B:=\sum_{j=1}^\infty\frac1{b_j}<\infty\quad\textnormal{and}\quad \alpha^*:=\liminf_{j\to\infty}\frac{\log\,a_j}j>0.
\end{align*}
Then the exponent $\tau^*$ of EC-SPT satisfies
\begin{align*}
\max\left(B,\frac{\log\,2}{\alpha^*}\right)\le \tau^*\le B+\frac{\nu}{\alpha^*},
\end{align*}
where $\nu=\log\,2$ for $\lall$ and $\nu=\log\,3$ for $\lstd$.
In particular, if $\alpha^* =\infty$ then $\tau^* = B$.
\end{enumerate}
\end{theorem}
\begin{proof}
{\bf Proof for the class $\lall$:} For the class $\lall$ the theorem was proven in a more general setting in \cite{IKPW15}.

\vspace{0.25cm}
\noindent {\bf Proof of Point \ref{it:appexp} for the class $\lstd$:}
To show EXP we will use \eqref{eqboundMapp2}. For $s\in \NN$ and $\eps\in(0,1)$ define
\begin{align*}
m=\max_{j=1,2,\dots,s}\ \left\lceil \left(\frac{2^{b_j}}{a_j}\,\frac{\log\left(1+\frac{sR}{\log(1+\eta^2)}\right)}{\log\,\omega^{-1}}\right)^{B(s)}\,\right\rceil,
\end{align*}
where
\begin{align*}
\eta=\left(\frac{\eps^2}{2D(s,\omega,\bsb)^{\frac{1}{B(s)+2}}}\right)^{\frac{B(s)+2}{2}}\qquad\textnormal{and}\qquad R = \max_{1 \le j\le s} \sum_{h=1}^\infty \omega^{a_j 2^{-b_j} (h^{b_j}-1)}<\infty.
\end{align*}
Let $n=\prod_{j=1}^s n_j$ with $n_j:=\left\lfloor m^{1/(B(s)  b_j)}\right\rfloor$ for $j=1,2,\ldots,s$. Now we can write for $F_n$, given by \eqref{eq:deffn},
\begin{align*}
F_n=\sum_{\bsl \in \cG_{n,s}^{\bot} \setminus \{\bszero\}}\omega^{\sum_{j=1}^s a_j\abs{\ell_j/2}^{b_j}}=-1+\prod_{j=1}^s\left(1+\sum_{h=1}^\infty \omega^{a_j 2^{-b_j}(2n_j h)^{b_j}}\right).
\end{align*}
Since $\lfloor x\rfloor\ge x/2$ for all $x\ge1$, we have
\begin{align*}
(2n_jh)^{b_j}\ge h^{b_j}\,m^{1/B(s)}\qquad\textnormal{for all}\qquad j=1,2,\dots,s.
\end{align*}
Hence,
\begin{align*}
F_n \le -1+\prod_{j=1}^s\left(1+\sum_{h=1}^{\infty} \omega^{m^{1/B(s)} a_j\, (h/2)^{b_j}}\right).
\end{align*}
We further estimate
\begin{align*}
\sum_{h=1}^{\infty} \omega^{m^{1/B(s)} a_j\,(h/2)^{b_j}}&=\omega^{m^{1/B(s)} a_j\,2^{-b_j}}\sum_{h=1}^{\infty} \omega^{m^{1/B(s)} a_j\,2^{-b_j}\,(h^{b_j}-1)}
\le \omega^{m^{1/B(s)} a_j 2^{-b_j}}R.
\end{align*}
From the definition of $m$ we have $\omega^{m^{1/B(s)} a_j 2^{-b_j}}R\leq\log(1+\eta^2)/s$ for all $j=1,2,\dots,s$. This proves
\begin{align}\label{eqFnbound}
F_n\le-1+\left(1+\frac{\log(1+\eta^2)}{s}\right)^s\le -1+\exp(\log(1+\eta^2))=\eta^2.
\end{align}
Now, plugging this into \eqref{eqboundMapp2}, we obtain
\begin{align}\label{eqboundMapp3}
e(A_{n,s,M},\APP_s;\lstd)^2\le\frac{1}{M}+ M^{B(s)+1} D(s,\omega,\bsb) \eta^2.
\end{align}
Because of the choice of $\eta$ we can choose $M$ as
\begin{align*}
M=\frac{1}{D(s,\omega,\bsb)^{\frac{1}{B(s)+2}} \eta^{\frac{2}{B(s)+2}}}=\frac{2}{\eps^2}\ge 1,
\end{align*}
which yields, inserting into \eqref{eqboundMapp3},
\begin{align*}
e(A_{n,s,M},\APP_s;\lstd)^2\leq 2D(s,\omega,\bsb)^{\frac{1}{B(s)+2}} \eta^{\frac{2}{B(s)+2}}=\eps^2.
\end{align*}
Furthermore, note that
\begin{align*}
n=\prod_{j=1}^sn_j=\prod_{j=1}^s \left\lfloor m^{1/(B(s) b_j)} \right\rfloor \le m^{\frac{1}{B(s)}\sum_{j=1}^s 1/b_j} = m=\cO\left(\log^{B(s)}\left (1 +\eta^{-1}\right)\right),
\end{align*}
as $\eta$ tends to zero, with the factor in the $\cO$ notation independent of $\eps^{-1}$ but dependent on $s$, see \cite[Proof of Theorem~3]{KPW14}. Now it follows that
\begin{align*}
n(\eps,\APP_s;\lstd)=\cO\left(\log^{B(s)}\left (1 +\eps^{-1}\right)\right).
\end{align*}
This implies that we indeed have exponential convergence for $\lstd$ for all $\bsa$ and $\bsb$, with $p(s)=1/B(s)$, and thus $p^* (s)\geq 1/B(s)$. On the other hand, note that obviously $e(n,\APP_s;\lstd)\geq e(n,\APP_s;\lall)$, hence the rate of exponential convergence for $\lstd$ cannot be larger than for $\lall$, which is $1/B(s)$. Thus, also for the class $\lstd$ we have $p^*(s)= 1/B(s)$.

\vspace{0.25cm}
\noindent {\bf Proof of Point \ref{it:appuexp} for the class $\lstd$:} Suppose first that $\bsa$ is an arbitrary sequence and that $\bsb$ is such that $B=\sum_{j=1}^\infty 1/b_j<\infty$. Then we can derive in the same way as above, by replacing $B(s)$ by $B$, that
\begin{align*}
n(\eps,\APP_s;\lstd)=\cO\left(\log^{B}\left (1 +\eps^{-1}\right)\right),
\end{align*}
hence UEXP with $p^*\geq1/B$ holds. On the other hand, if we have uniform exponential convergence for $\lstd$, this implies UEXP for $\lall$, which in turn implies that $B<\infty$ and that $p^*\leq 1/B$.

\vspace{0.25cm}
\noindent {\bf Proof of Point \ref{it:appwt} for the class $\lstd$:} This assertion can be proven in the same way as in \cite[Subsection 8.2]{DKPW13} and \cite[Proof of Point 3]{IKPW15a}, because these proofs just depend on the form of the weights $\omega_{\bsk}$.

\vspace{0.25cm}
\noindent {\bf Proof of Point \ref{it:appspt} for the class $\lstd$:} Suppose that EC-PT holds for the class $\lstd$. Then EC-PT also holds for the class $\lall$, which in turn implies that $B<\infty$ and $\alpha^*>0$. Moreover, it is trivial that EC-SPT implies EC-PT. So it remains to show the sufficiency of the conditions on $\bsa$ and $\bsb$ for EC-SPT.

To this end, we analyze the algorithm $A_{n,s,M}$ given by~\eqref{eqdefAnsM}, where the sample points $\bsx_k$ form a midpoint rule with
\begin{align*}
n_j=2\,\left\lceil\,\left(\frac{\log\,M}{a_j^{\beta}\log\,\omega^{-1}}\right)^{1/b_j}\right\rceil\,-\,1\qquad\textnormal{for all}\qquad j=1,2,\dots,s.
\end{align*}
Here $M>1$ and $\beta\in(0,1)$. Note that $n_j\ge1$ and is always an odd number. Furthermore $n_j=1$ if $a_j\ge ((\log M)/(\log \omega^{-1}))^{1/\beta}$. Assume that $\alpha^* \in (0,\infty]$. Since for all $\delta\in(0,\alpha^*)$ we have for some $j_\delta^*\in\NN$,
\begin{align*}
a_j\ge \exp(\delta j) \quad\textnormal{for all}\quad j\ge j^*_\delta,
\end{align*}
see~\cite[Eq. (1)]{DKPW13}, we conclude that
\begin{align*}
j\ge j^{*}_{\beta,\delta,M}:=\max\left(j^*_\delta,\frac{\log(((\log M)/(\log\omega^{-1}))^{1/\beta})}{\delta}\right)\quad\textnormal{implies}\quad n_j=1.
\end{align*}


Similar to \cite[p.\ 25]{IKPW15a} it can be shown that for this choice of the $n_j$ it cannot happen that there exist $\bsl^{(1)}, \bsl^{(2)}\in\cG_{n,s}^\perp$ with $\bsl^{(1)}\neq \bsl^{(2)}$ such that $\bsh(\pm)_{\uu}\bsl^{(1)}=\bsh(\pm)_{\uu}\bsl^{(2)}$ for any $\uu\subseteq[s]$ and $\bsh\in\cA (s,M)$. Consequently, each coefficient $\widetilde{f}(\bsh(\pm)_{\uu}\bsl)$ in \eqref{errfhab2} occurs at most once and so we obtain
\begin{align}\label{errfhab3}
\left|\int_{[0,1]^s} f_{\bsh}(\bsx)\rd\bsx-\frac{1}{n}\sum_{k=1}^{n} f_{\bsh}(\bsx_k) \right|^2 \le \|f\|^2_{\Kcos} \left(\frac{1}{2^s} \sum_{\uu \subseteq [s]} \sum_{\bsl \in \cG_{n,s}^{\bot} \setminus \{\bszero\}} 2^{|\bsl|_*} \omega_{\bsh (\pm)_{\uu} \bsl}\right).
\end{align}
Thus, \eqref{errfhab3} together with \eqref{approx_err}, \eqref{approx_err1} gives
\begin{align*}
e_{n,s}^2:=e(A_{n,s,M},\APP_s;\lstd)^2\le \frac1M\,+\,\sum_{\bsh\in\cA(s,M)}\ \frac{1}{2^s} \sum_{\vv\subseteq [s]}\sum_{\bsl\in\cG_{n,s}^\perp\setminus\{\bszero\}}2^{\abs{\bsl}_*}\omega_{\bsh (\pm)_{\vv}\bsl}.
\end{align*}

We now estimate, for a fixed $\vv\subseteq [s]$,
\begin{align*}
\sum_{\bsl \in \cG^\perp_{n,s} \setminus \{\bszero\} } 2^{\abs{\bsl}_*}\omega_{\bsh (\pm)_{\vv}\bsl}=  \sum_{\emptyset \neq \uu \subseteq [s]} \prod_{j\in \uu}\left( \sum_{\ell =1}^\infty 2\omega^{a_j(h_j (\pm)_{\vv}2\ell n_j)^{b_j}}\right)\, \prod_{j \not \in \uu}\omega^{a_jh_j^{b_j}},
\end{align*}
where we separated the cases for $\ell \neq 0$ and $\ell=0$, and where
\begin{align*}
h_j (\pm)_{\vv}2\ell n_j=\begin{cases} h_j + 2\ell n_j & \textnormal{if } j \in \vv,\\  |h_j-2 \ell n_j| & \textnormal{if } j \not\in \vv.\end{cases}
\end{align*}
We estimate the second product by one so that
\begin{align*}
\sum_{\bsl \in \cG^\perp_{n,s} \setminus \{\bszero\} } 2^{\abs{\bsl}_*}\omega_{\bsh (\pm)_{\vv}\bsl}\le \sum_{\emptyset \neq \uu \subseteq [s]} \prod_{j\in \uu}\left( \sum_{\ell =1}^\infty 2\omega^{a_j(h_j (\pm)_{\vv}2\ell n_j)^{b_j}}\right).
\end{align*}
Recall that for $\bsh\in\cA(s,M)$ we have $h_j<(n_j+1)/2$ for all $j=1,2,\dots,s$. In particular, if $n_j=1$ then $h_j=0$ and
\begin{align}\label{ineqhone}
\sum_{\ell=1}^\infty 2\omega^{a_j(h_j (\pm)_{\vv}2\ell n_j)^{b_j}} =2\sum_{\ell=1}^\infty\omega^{a_j(2\ell)^{b_j}}\le 2\omega^{a_j}\sum_{\ell=1}^\infty\omega^{a_j\left((2\ell)^{b_j} -1\right)}
\le 2\omega^{a_j}\sum_{\ell=1}^\infty\omega^{a_*\left((2\ell)^{b_*} -1\right)}  = 2\omega^{a_j}A_1,
\end{align}
where $A_1$ is a constant independent of $s$.

Let $n_j\ge3$. Then  $h_j<(n_j+1)/2$. Since $h_j\ge 0$ and $(n_j+1)/2$ is a positive integer, we conclude that $h_j\le (n_j+1)/2-1=(n_j-1)/2$, and so, since $\ell\neq 0$,
\begin{align*}
h_j (\pm)_{\vv}2\ell n_j= \begin{cases} h_j + 2\ell n_j & \mbox{if }j\in\vv,\\ 2\ell n_j - h_j & \mbox{if }j\notin\vv.\end{cases}
\end{align*}
In any case we obtain
\begin{align*}
h_j (\pm)_{\vv}2\ell n_j\ge 2\ell n_j - h_j\ge 2\ell n_j - \frac{n_j -1}{2} \ge \frac{n_j +1}{2} \ell.
\end{align*}
Therefore, as $n_j\ge 3$,
\begin{align}\label{ineqhtwo}
\sum_{\ell =1}^\infty 2\omega^{a_j(h_j (\pm)_{\vv}2\ell n_j)^{b_j}}&\le 2\sum_{\ell=1}^\infty\omega^{a_j[(n_j+1)/2]^{b_j}\ell^{b_j}}\nonumber\\
&\le\omega^{a_j[(n_j+1)/2]^{b_j}} 2\sum_{\ell=1}^\infty \omega^{a_j[(n_j+1)/2]\left(\ell^{b_j}-1\right)}\nonumber\\
&\le\omega^{a_j[(n_j+1)/2]^{b_j}} 2A_2,
\end{align}
where $A_2$ is another constant independent of $s$.

The inequalities \eqref{ineqhone} and \eqref{ineqhtwo} can be combined to
\begin{align*}
\beta_j:=\sum_{\ell =1}^\infty 2\omega^{a_j(h_j (\pm)_{\vv}2\ell n_j)^{b_j}}\le \omega^{a_j[(n_j+1)/2]^{b_j}} A,
\end{align*}
where $A=2 \max\{A_1,A_2\}$ is a constant independent of $s$.

Note that
\begin{align*}
\sum_{\emptyset \neq \uu \subseteq [s]} \prod_{j\in \uu}\left(\sum_{\ell=1}^\infty 2\omega^{a_j(h_j (\pm)_{\vv}2\ell n_j)^{b_j}}\right)=-1+\sum_{\uu \subseteq [s]} \prod_{j\in \uu} \beta_j=-1+\prod_{j=1}^s(1+\beta_j).
\end{align*}
Consequently,
\begin{align*}
e_{n,s}^2\le \frac{1}{M}+ |\cA(s,M)| \left(-1 + \prod_{j=1}^s \left(1 + \omega^{a_j[(n_j+1)/2]^{b_j}} A\right)\right).
\end{align*}
Using $\log (1+x) \le x$ we obtain
\begin{align*}
\log\left[ \prod_{j=1}^s \left(1 + \omega^{a_j[(n_j+1)/2]^{b_j}} A \right)\right] \le  A \sum_{j=1}^s\omega^{\,a_j[(n_j+1)/2]^{b_j}}=:\gamma.
\end{align*}
From the definition of $n_j$ we have $a_j[(n_j+1)/2]^{b_j}\ge a_j^{1-\beta}\,(\log\,M)/\log\,\omega^{-1}$. Therefore
\begin{align*}
\omega^{a_j[(n_j+1)/2]^{b_j}}\le \omega^{a_j^{1-\beta}\,(\log\,M)/\log\,\omega^{-1}}= \left(\frac1M\right)^{a_j^{1-\beta}}.
\end{align*}

Without loss of generality, we assume $M\ge 2$. Since $a_j\ge 1$ for $j\le j^*_{\beta,\delta,M}-1$ and $a_j\ge \exp(\delta j)$ for $j\ge j^*_{\beta,\delta,M}$ we obtain
\begin{align*}
\gamma\le A\,\left(\frac{j^*_{\beta,\delta,M}-1}{M}+\sum_{j=j^*_{\beta,\delta,M}}^\infty\left(\frac1M\right)^{\exp((1-\beta) \delta j)}\right).
\end{align*}
Note that there exists a constant $C>0$ such that $j^*_{\beta,\delta,M}\leq (\log\log M)\,j^*_{\beta,\delta}$ with $j^*_{\beta,\delta}:=C \max(j^*_{\delta}, (1-\log\log \tilde{\omega}^{-1})/(\delta\beta))$. Thus
\begin{align*}
\gamma\le A\,\left(\frac{(\log\log M)j^*_{\beta,\delta}-1}{M}+\sum_{j=0}^\infty\left(\frac{1}{M}\right)^{\exp((1-\beta) \delta j)}\right)\le\frac{C_{\beta,\delta}}{M^{1/2}},
\end{align*}
with
\begin{align*}
C_{\beta,\delta}:=A\,\left(j^*_{\beta,\delta}-1+\sum_{j=0}^\infty\left(\frac{1}{2}\right)^{\exp((1-\beta) \delta j)-1}\right)<\infty,
\end{align*}
where we made use of $M\ge 2$. Note that for $M\ge C_{\beta,\delta}^2$ we have $\gamma\le 1$.

Using convexity we easily check that $-1+\exp(\gamma) \le (\mathrm{e}-1)\gamma$ for all $\gamma\in[0,1]$. Thus for $M\ge C_{\beta,\delta}^2$ we obtain
\begin{align*}
-1 + \prod_{j=1}^s \left(1 +A\,\omega^{\,a_j\left(\frac{n_j+1}{2}\right)^{b_j}}\right) \le -1+\exp(\gamma)\le (\mathrm{e}-1)\gamma\le\frac{C_{\beta,\delta}\,(\mathrm{e}-1)}{M^{1/2}}.
\end{align*}
We now turn to $|\cA(s,M)|$. From the proof of \cite[Theorem 9]{IKPW15} we obtain
\begin{align*}
|\cA(s,M)|\le2^{j^*_{\beta,\delta,M}} \left(1+\frac{\log\,M}{\log\,\omega^{-1}}\right)^{B+(\log 2)/\delta}.
\end{align*}
Therefore
\begin{align*}
e^2_{n,s}\le \frac{1}{M}\,\left[1+C_{\beta,\delta}(\mathrm{e}-1)2^{j^*_{\beta,\delta,M}} \left(1+\frac{\log\,M}{\log\,\omega^{-1}}\right)^{B+(\log 2)/\delta}\right]\le \frac{D_{\beta,\delta}}{M^{1/2}},
\end{align*}
where
\begin{align*}
D_{\beta,\delta}:=\sup_{x\ge C_{\beta,\delta}}\left(\frac1{x^{1/2}}+\frac{C_{\beta,\delta}(\mathrm{e}-1)(\log x)^{j^*_{\beta,\delta}\log 2}}{x^{1/2}}\,\left(1+\frac{\log\,x}{\log\,\omega^{-1}}\right)^{B+(\log 2)/\delta}\right)<\infty.
\end{align*}
Hence for $M=\max\left(C_{\beta,\delta}^2,D_{\beta,\delta}^{2}\,\eps^{-4},2\right)$ we have $e_{n,s}\leq\eps$. Now we estimate the number $n$ of function values used by the algorithm $A_{n,s,M}$. We have
\begin{align*}
n&=\prod_{j=1}^s n_j=\prod_{j=1}^{\min(s,j^*_{\beta,\delta,M})}n_j\le \prod_{j=1}^{\min(s,j^*_{\beta,\delta,M})}\left(1+2\left(\frac{\log\,M}{a_j^{\beta}\,\log\,\omega^{-1}}\right)^{1/b_j}\right)\\
&\le 3^{j^*_{\beta,\delta}}\,\left(\frac{\log\,M}{\log\,\omega^{-1}}\right)^B=\cO\left(\left(1+\log\,\eps^{-1}\right)^{B+(\log 3)/(\beta\,\delta)}\right),
\end{align*}
where the factor in the big $\cO$ notation depends only on $\beta$ and $\delta$. This proves SPT with $\tau=B+\log 3 /(\beta\,\delta)$. Since $\beta$ can be arbitrarily close to one, and $\delta$ can be arbitrarily close to $\alpha^*$,the exponent $\tau^*$of SPT is at most $B+\frac{\log 3}{\alpha^*}$, where for $\alpha^*=\infty$ we have $\frac{\log 3}{\alpha^*} = 0$.

\end{proof}

\begin{remark}
The only point where we need the monotonicity of the sequence $\bsa$ is where we show that EC-PT implies $\alpha^{\ast}>0$.
\end{remark}

\begin{small}
\noindent\textbf{Authors' addresses:}
\\ \\
\noindent Christian Irrgeher, Friedrich Pillichshammer,
\\
Department of Financial Mathematics and Applied Number Theory,
Johannes Kepler University Linz, Altenbergerstr.~69, 4040 Linz, Austria\\
\\
\noindent Peter Kritzer,
\\
Johann Radon Institute for Computational and Applied Mathematics (RICAM),
Austrian Academy of Sciences, Altenbergerstr.~69, 4040 Linz, Austria.\\

\noindent \textbf{E-mail:} \\
\texttt{christian.irrgeher(AT)jku.at}\\
\texttt{peter.kritzer(AT)oeaw.ac.at}\\
\texttt{friedrich.pillichshammer(AT)jku.at} \\
\end{small}


\begin{thebibliography}{00}

\bibitem{DP10} J.~Dick and F.~Pillichshammer: \textit{Digital Nets and Sequences. Discrepancy Theory and uasi-Monte Carlo Integration}. Cambridge University Press, Cambridge, 2010.

\bibitem{DKPW13} J.~Dick, P.~Kritzer, F.~Pillichshammer, H.~Wo\'{z}niakowski: Approximation of analytic functions in Korobov spaces. J. Complexity 30: 2--28, 2014.

\bibitem{DLPW11} J.~Dick, G.~Larcher, F.~Pillichshammer, H.~Wo\'{z}niakowski: Exponential convergence and tractability of multivariate integration for Korobov spaces. Math. Comp. 80: 905--930, 2011.

\bibitem{DNP14} J.~Dick, D.~Nuyens, F.~Pillichshammer: Lattice rules for nonperiodic smooth integrands. Numer. Math. 126, 259--291, 2014.

\bibitem{IKLP15} C.~Irrgeher, P.~Kritzer, G.~Leobacher, and F.~Pillichshammer: Integration in Hermite space of analytic functions. J. Complexity 31: 380--404, 2015.

\bibitem{IKPW15} C.~Irrgeher, P.~Kritzer, F.~Pillichshammer and H.~Wo\'{z}niakowski: Tractability of Multivariate Approximation defined over Hilbert spaces with Exponential Weights. submitted, 2015.  See arXiv:1502.03286

\bibitem{IKPW15a} C.~Irrgeher, P.~Kritzer, F.~Pillichshammer and H.~Wo\'{z}niakowski: Approximation in Hermite spaces of smooth functions. Submitted, 2015. See arXiv:1506.08600


\bibitem{KPW14} P.~Kritzer, F.~Pillichshammer, and H.~Wo\'{z}niakowski: Multivariate Integration of infinitely many times differentiable functions in weighted Korobov spaces. Math. Comp. 83: 1189--1206, 2014.

\bibitem{KPW14a} P.~Kritzer, F.~Pillichshammer, and H.~Wo\'{z}niakowski: Tractability of multivariate analytic problems. In: {\it Uniform Distribution and Quasi-Monte Carlo Methods. Discrepancy, Integration and Applications} (P. Kritzer, H. Niederreiter, F. Pillichshammer and A. Winterhof, eds.). De Gruyter, Berlin, 2014.

\bibitem{KSW06} F.Y.~Kuo, I.H.~Sloan, H.~Wo\'{z}niakowski. Lattice rules for multivariate approximation in the worst case setting. In: \textit{Monte Carlo and Quasi-Monte Carlo Methods 2004} (H. Niederreiter, D. Talay, eds.). Springer, Berlin, 2006, 289--330.

\bibitem{NSW04} E.~Novak, I.H.~Sloan, H.~Wo\'{z}niakowski. Tractability of approximation for weighed Korobov spaces on classical and quantum computers. Found. Comput. Math. 4, 121--156, 2004.

\bibitem{NW08} E.~Novak and H.~Wo\'zniakowski: \textit{Tractability of Multivariate Problems, Volume I: Linear Information}. EMS, Zurich, 2008.

\bibitem{NW10} E.~Novak and H.~Wo\'zniakowski: \textit{Tractability of Multivariate Problems, Volume II: Standard Informations for Functionals}. EMS, Zurich, 2010.

\bibitem{NW12} E.~Novak and H.~Wo\'zniakowski: \textit{Tractability of Multivariate Problems, Volume III: Standard Informations for Operators}. EMS, Zurich, 2012.

\bibitem{TWW88} J.F.~Traub, G.W.~Wasilkowski, and H.~Wo\'zniakowski: \textit{Information-Based Complexity.} Academic Press, New York, 1988.

\end{thebibliography}
\end{document}